\documentclass[oneside,english]{amsart}
\usepackage[T1]{fontenc}
\usepackage[latin9]{inputenc}
\usepackage{color}
\usepackage{textcomp}
\usepackage{amstext}
\usepackage{amsthm}
\usepackage{amsmath}
\usepackage{amssymb}
\usepackage{esint}
\usepackage{soul}
\usepackage[backend=bibtex8,maxbibnames=99,style=alphabetic]{biblatex} 
\bibliography{Biblio_jabref.bib} 
\DeclareNameAlias{default}{last-first}

\AtEveryBibitem{\clearfield{month}
\clearfield{day}
\clearfield{doi}
\clearfield{url}
\clearfield{issn}
}
\renewbibmacro{in:}{}
\DeclareFieldFormat
  [article,inbook,incollection,inproceedings,patent,thesis,unpublished]
  {title}{#1\isdot}
\makeatletter
\def\url@leostyle{%
  \@ifundefined{selectfont}{\def\UrlFont{\sf}}{\def\UrlFont{\small\sffamily}}}
\makeatother
\makeatletter
\numberwithin{equation}{section}
\numberwithin{figure}{section}
\theoremstyle{plain}
\newtheorem{lettertheorem}{\protect\theoremname}
  \theoremstyle{plain}
  \newtheorem{letterlemma}[lettertheorem]{Lemma}
  \theoremstyle{plain}
\newtheorem{thm}{\protect\theoremname}
  \theoremstyle{plain}
  
  \theoremstyle{plain}
  
  \theoremstyle{plain}
    \newtheorem{defi}{Definition}
   
  \theoremstyle{plain}
     
  \theoremstyle{plain}
  
  \theoremstyle{remark}
  \newtheorem{rem}{\protect\remarkname}

\DeclareMathOperator{\diam}{diam}

\DeclareMathOperator{\loc}{loc}

\DeclareMathOperator{\RH}{RH}

\def\Xint#1{\mathchoice
{\XXint\displaystyle\textstyle{#1}}%
{\XXint\textstyle\scriptstyle{#1}}%
{\XXint\scriptstyle\scriptscriptstyle{#1}}%
{\XXint\scriptscriptstyle%
\scriptscriptstyle{#1}}%
\!\int}
\def\XXint#1#2#3{{\setbox0=\hbox{$#1{#2#3}{%
\int}$ }
\vcenter{\hbox{$#2#3$ }}\kern-.6\wd0}}

\def\dashint{\thinspace \Xint-}

\usepackage{pgf,tikz}
\usetikzlibrary{arrows}
\usepackage{url}

\makeatother

\usepackage{babel}
  \providecommand{\lemmaname}{Lemma}
   \providecommand{\corollaryname}{Corollary}
  \providecommand{\remarkname}{Remark}
\providecommand{\theoremname}{Theorem}

\usepackage[
  colorlinks,
  linkcolor=purple
]{hyperref}
\newcommand\rurl[2]{%
  \href{http://#1}{\nolinkurl{#1}}%
}

\providecommand{\keywords}[1]{\textbf{\textit{Index terms---}} #1}

\begin{document}

\title[A note on weighted improved Poincar\'e-type inequalities]{A note on generalized Poincar\'e-type inequalities with applications to weighted improved Poincar\'e-type inequalities}

\keywords{Improved Poincar\'e-Sobolev inequalities, weights, fractional inequalities}

\subjclass[2010]{Primary: 35A23, and Secondary: 42B20} 


\author[J.C. Mart\'{i}nez-Perales]{Javier C. Mart\'{i}nez-Perales}
\address[J.C. Mart\'{i}nez-Perales]{
BCAM\textendash  Basque Center for Applied Mathematics, Bilbao, Spain}
\email{jmartinez@bcamath.org}


\begin{abstract}
The main result of this paper supports a conjecture by C. P\'erez and E. Rela about the properties of the weight appearing in their main result in \cite{Perez2018}. The result we obtain does not need any condition on the weight, but still is not fully satisfactory, even though the result in \cite{Perez2018} is obtained as a corollary of ours. Also, we extend the conclusions of their theorem to the range $p<1$.

As an application of our result, we give a unified vision of weighted improved Poincar\'e-type inequalities in the Euclidean setting, which gathers both weighted improved classical and fractional Poincar\'e inequalities within an approach which avoids any representation formula. We obtain results in the direction of those in \cite{Drelichman2008} and \cite{Cejas2019} and furthermore we improve them in some aspects.

Finally, we also explore analog inequalities in the context of metric spaces by means of the already known self-improving results.
\end{abstract}

\maketitle

\section{Introduction}

Recently, in \cite{Perez2018}, the authors consider a locally integrable function satisfying in every cube $Q$ the starting inequality 
\begin{equation}\label{eq:starting_point}
\dashint_Q |f(x)-f_Q|dx\leq a(Q),
\end{equation}
where the dashed integral represents the average with respect the underlying measure, $f_Q:=\dashint_Qf(x)dx$, and $a:\mathcal{Q}\to[0,\infty)$ is some functional defined on the family of all cubes in $\mathbb{R}^n$. Then, for a given value $p\geq 1$, they are able to get the self-improved inequality 
\begin{equation}\label{mejora1}
\left(\dashint_Q|f-f_Q|^pdw\right)^{1/p}\leq C_ns\|a\|^sa(Q),
\end{equation}
as long as the functional $a$ satisfies the so-called $SD_p^s(w)$ condition, namely, for any disjoint subfamily $\{Q_j\}_j$ of cubes contained in $Q$,
\[
 \sum_ja(Q_j)^pw(Q_j)\leq C^p\left(\frac{\left|\bigcup_jQ_j\right|}{|Q|}\right)^{p/s}a(Q)^pw(Q),
 \]
where $C>0$ is a constant independent of the cubes we consider and $w$ is in the class $A_\infty$ of all Muckenhoupt   weights. The authors remark that, although the $A_\infty$ condition is assumed, the $A_\infty$ constant, which is defined by
\begin{equation}\label{Ainfty_constant}
[w]_{A_\infty}:=\sup_{Q\in\mathcal{Q}}\frac{1}{w(Q)} \int_Q M(w\chi_Q)dx,
\end{equation}
does not appear in the result. Then,  they conjecture that the $A_\infty$ condition is not needed  and  that the result should hold for any weight.  This conjecture is supported by the fact that, for the functional $a(Q):=\ell(Q)^\alpha$, $\alpha>0$   (which for any weight $w$ satisfies the smallness condition $SD_p^s(w)$ for some $s>1$, and for any  $p\geq 1$) which satisfies that the self-improved inequality \eqref{mejora1}  holds for any weight $w$. Indeed, it is known that, when $0<\alpha\leq 1$, the class of functions satisfying inequality \eqref{eq:starting_point} for $a(Q)=C\ell(Q)^\alpha$ is precisely the H\"older-Lipschitz class $\Lambda_\alpha$ (see \cite{Campanato1963}). Thus, no matter the weight $w$ we take nor the value $p\geq1$ we consider, the following argument can be performed
\[
\begin{split}
\left(\dashint_Q|f-f_Q|^pdw\right)^{1/p}& \leq \left(\dashint_Q\dashint_Q|f(x)-f(y)|^pdydw(x)\right)^{1/p}\\
&\lesssim \left(\dashint_Q\dashint_Q|x-y|^{p\alpha}dydw(x)\right)^{1/p}\\
&\lesssim \ell(Q)^\alpha\left(\dashint_Q\dashint_Q dydw(x)\right)^{1/p}\\
&= \ell(Q)^\alpha.
\end{split}
\]
 This suggests that the $A_\infty$ condition is somehow an artifice of the proof, and it seems that one should be able to get rid of it.

The purpose of this paper is twofold. On one hand, we will get a generalization of the main result in \cite{Perez2018} which is valid for any weight $w$ regardless of the properties it satisfies. We actually get a result in which, instead of obtaining an improvement with a weighted $L^p(Q,dw)$ average, we get the following improvement
\begin{equation}\label{this_result}
\left(\frac{1}{|Q|\left(\dashint_Qw^r\right)^{1/r}}\int_Q |f-f_Q|^pdw\right)^{1/p}\lesssim a(Q),
\end{equation}
where $r\geq 1$ is some value for which $a$ satisfies that, given a disjoint family $\{Q_j\}_j$ of subcubes of a cube $Q$,
\begin{equation}\label{weighted_smallness}
 \sum_ja(Q_j)^p|Q_j|\left(\dashint_{Q_j}w^r\right)^{1/r}\leq \|a\|^p\left(\frac{\left|\bigcup_jQ_j\right|}{|Q|}\right)^{p/s}a(Q)^p|Q|\left(\dashint_Qw^r\right)^{1/r}.
 \end{equation}
Recall that, as was proved in \cite{Hytoenen2013,Hytoenen2012a}, for a weight $w\in A_\infty$ one has the existence of some $r>1$ for which $w\in \RH_r$. More precisely, for $r_w:=1+\frac{1}{2^{n+1}[w]_{A_\infty}-1}$, one has
\[
\left(\dashint_Q w^{r_w}\right)^{1/r_w}\leq 2\dashint_Q w,\quad\text{for every cube }Q.
\]

Jensen's inequality proves then that also $w\in \RH_r$ for every $r<r_w$ with the same constant, and so, condition \eqref{weighted_smallness} is equivalent to the $SD_p^s(w)$ condition for every $1\leq r\leq r_w$. The independence of \eqref{this_result} on $r$ implies that our result contains that in \cite{Perez2018}.
 
Even though we do not get a weighted average in our estimate, the result   will allow to prove several important results, which leads us to the second purpose of the paper: the obtaining of a unified approach to prove weighted improved Poincar\'e-type inequalities both in the classical and the fractional setting in John domains of the Euclidean space by avoiding the use of representation formulas.

More precisely, we are interested in the study on John domains of inequalities of the form
\begin{equation}\label{Def.Poincare}
\inf_{a\in \mathbb{R}}\|u-a\|_{L^{q}(\Omega,w )}\lesssim [u]_{W_{\tau}^{s,p}\left(\Omega,v  \right)},
\end{equation}	
where  $1\leq p,q< \infty$, $s,\tau\in(0,1]$, $w,v$ are weights and the notation
\begin{equation}\label{fractional_seminorm}
[u]_{W_{\tau}^{s,p}\left(\Omega,v \right)}:=\left(\int_\Omega\int_{\Omega\cap B(y,\tau d(y))}\frac{|u(x)-u(y)|^p}{|x-y|^{n+sp}} v(y)dxdy\right)^{1/p}
\end{equation}
is used whenever $s<1$ (the parameter $v$ will be omitted whenever its value is $1$).  Observe that by understanding $[u]_{W_{\tau}^{1,p}\left(\Omega,v \right)}$ as the classical seminorm of the Sobolev space $W^{1,p}(\Omega)$ for any $\tau\in(0,1)$, we obtain a unified approach for both classical and fractional weighted Poincar\'e-type inequalities. 

Recall that, roughly speaking, $\Omega$ is a John domain if it has a central point such that any other point can be connected to it without getting too close to the boundary (see Section \ref{sec:preliminaries} for a precise definition). These domains are essentially the largest class of domains in $\mathbb{R}^n$ for which the Sobolev-Poincar\'e inequality 
\begin{equation}\label{sobolev poincare}
\|u-u_{\Omega}\|_{L^\frac{np}{n-p}(\Omega)} \leq C \left(\int_{\Omega} |\nabla u(x)|^p \,dx \right)^{\frac{1}{p}},
\end{equation} 
 holds (see \cite{Mazya2011,Reshetnyak1980,Bojarski1988, Hajlasz2001} for the sufficiency, and \cite{Buckley1995} for the necessity), where $C$ does not depend on $u$ nor on the size of the domain $\Omega$. Here,  $u$ is a locally Lipschitz function, $1\leq p < n$ and $u_\Omega$ is the average of $u$ over $\Omega$.

The above inequality, also called $(\frac{np}{n-p}, p)$-Poincar\'e inequality, is a special case of a larger family of so-called improved Poincar\'e inequalities, which are $(q,p)$-Poincar\'e inequalities with a weight that is a power of the distance to the boundary $d(x)$, namely,
$$
\|u-u_{\Omega}\|_{L^q(\Omega)} \leq C \|d^{\alpha}|\nabla u|\|_{L^p(\Omega)}
$$
where $1\le p\le q \le \frac{np}{n-p(1-\alpha)}$,  $p(1-\alpha)<n$ and $\alpha \in [0,1]$ (see \cite{Boas1988,Hurri-Syrjaenen1994}, and also \cite{Drelichman2008,Acosta2017} for weighted versions). 

A classical technique for getting this kind of inequalities is through the use of a representation formula in terms of a  fractional integral, as can be seen for instance in \cite{Drelichman2008,Hurri-Syrjaenen2013}. Another classical argument goes through the use of chains of cubes in    order to reduce the problem of finding an inequality in $\Omega$ to its counterpart on these cubes. An approach which avoids  the use of any representation formula  to obtain Poincar\'e-Sobolev inequalities on cubes (or balls) was introduced in \cite{Franchi1998} (and then sharpened in \cite{MacManus1998}). See also the recent work \cite{Perez2018} for more precise results on this direction. The local-to-global method began with the work \cite{Boman1982}  and later with \cite{Iwaniec1985,Jerison1986}, and has been used by many authors, for example \cite{Hurri-Syrjaenen1994,Chua1993} and \cite{Hurri1990}, where both the integral representation formula and the local-to-global methods are used.

It is also worth noting that these inequalities have also been studied in metric spaces with doubling measures, replacing $|\nabla u|$ by a generalized gradient (see  \cite{Hajlasz2000} and references therein).

In  recent years,  several authors have turned their attention to the fractional counterpart of inequality \eqref{sobolev poincare}, beginning with the work \cite{Hurri-Syrjaenen2013} where the inequality 
\begin{equation}\label{fraccionaria euclidea}
\|u-u_{\Omega}\|_{L^q(\Omega)} \leq C \left(\int_{\Omega}\int_{\Omega \cap B(x,\tau d(x))} \frac{|u(x)-u(z)|^p}{|x-z|^{n+sp}}\,dx\,dz\right)^{1/p} 
\end{equation}
was obtained for a bounded John domain $\Omega$, $s, \tau \in (0,1)$, $p<\frac{n}{s}$ and $1< p \leq q \leq \frac{np}{n-sp}$. The case $p=1$ was proved in \cite{Dyda2016} using  the so-called Maz'ya's truncation method (see \cite{Mazya2011}) adapted to the fractional setting, which allows to obtain a strong inequality from a weak one.  Alternatively, \eqref{fraccionaria euclidea} can be deduced by applying the main result in \cite{Perez2018}  and then using chains of cubes as mentioned above.

Observe that the seminorm appearing on the right hand side of inequality (\ref{fraccionaria euclidea}) is stronger than that of the usual fractional Sobolev space $W^{s,p}(\Omega)$. More precisely, if we consider $W^{s,p}(\Omega)$ to be the subspace of $L^p(\Omega)$ induced by the seminorm
$$
[f]_{W^{s,p}(\Omega)}:= \left(\int_{\Omega}\int_{\Omega} \frac{|f(x)-f(z)|^p}{|x-z|^{n+sp}}\,dx\,dz\right)^{1/p},
$$
then it is known that this space coincides with that defined by the unweighted seminorm $[f]_{W_{\tau}^{s,p}\left(\Omega  \right)}$ given in \eqref{fractional_seminorm}  when $\Omega$ is Lipschitz (see \cite{Dyda2006}), but there are examples of John domains $\Omega \subset \mathbb{R}^n$ for which the inclusion $W^{s,p}(\Omega)\subset  W^{s,p}_\tau(\Omega)$ is strict (see \cite{Drelichman2019} for this result and characterizations of both spaces as interpolation spaces). This has led to call inequality \eqref{fraccionaria euclidea} an ``improved'' fractional inequality. However, throughout this work, we will use this terminology to refer to inequalities including powers of the distance to the boundary as weights, as in the classical case. 
 
Improvements of an inequality like \eqref{fraccionaria euclidea}   were obtained in \cite{Drelichman2017} by including powers of the distance to the boundary  as weights on both sides of the  estimate, and also in \cite{LopezGarcia2017}, where the weights are defined by powers of the distance to a compact set of the boundary of the domain. Recently, in \cite{Cejas2019}, the authors have obtained improved fractional Poincar\'e-Sobolev inequalities on John domains of abstract metric spaces endowed with a measure which satisfies  some properties with respect to the metric.

 We will say that $\Omega$ supports the $(w,v)$-weighted fractional $(q,p)$-Poincar\'e inequality in $\Omega$ if \eqref{Def.Poincare} holds on $\Omega$ for every function $u\in W^{s,p}(\Omega,w)$ (when $s=1$ we omit the word  ``fractional'' in the notation). When $w$ and $v$ are defined  by including functions of the distance to the boundary in their expression, we shall refer to these inequalities as $(w,v)$-weighted improved  inequalities or just as $(w,v)$-improved  inequalities, when they are just defined by functions of the distance to the boundary (more functions than power functions are suitable in this approach).

Our results are in the spirit of a combination of the main results in \cite{Drelichman2008} and \cite{Cejas2019}. More precisely, we will improve the result in \cite{Cejas2019} by giving a weighted version of it and thus obtaining a fractional counterpart of the main result in \cite{Drelichman2008}. However, when restricted to the non-fractional setting, our results do not improve the main result in \cite{Drelichman2008}, but there is some overlap between both results, in the sense that, although we can state a more general version of it in terms of the functions of the distance to the boundary which are considered, we cannot obtain all the weights they get. We will stress the differences between our result and theirs in Section \ref{sec:applications}. Also, we have not been able to improve completely the results in \cite{Cejas2019}, as we cannot give weights defined by the distance to a compact set of the boundary instead of weights defined by the distance to the boundary, as it is done in \cite{LopezGarcia2017}.

The outline of the rest of the paper is as follows: we devote Section \ref{sec:preliminaries} to the statement of the main tools and previous results. In Section \ref{sec:improvement}  we prove the main result of this work, which extends Theorem 1.5 in \cite{Perez2018} and supports the non-$A_\infty$ conjecture. In Section \ref{sec:applications}, as an application of the main result of the paper, we give a unified approach for the study of fractional and non-fractional Poincar\'e inequalities in the Euclidean space. Section \ref{sec:metric} is dedicated to the study of some analog results in the more general setting of the spaces of homogeneous type.

\section{Preliminaries and previous results}\label{sec:preliminaries}

From now on $C$ and $c$ will denote constants that can change their values even within a single string of estimates. When necessary, we will stress the dependence of a constant on a particular parameter by writing it as a subindex. Also, we will use the notation $A\lesssim B$ whenever there exists a  constant $c>0$ independent of all relevant parameters for which $A\leq cB$.  Whenever $A\lesssim B$ and $B\lesssim A$, we will write $A\asymp B$. 

 The distance between a point $x$ and  the boundary of $\Omega$ will be denoted $d(x):=\inf_{y\in\partial\Omega}|x-y|$.  For given $r>0$ and $x\in \mathbb{R}^n$, the cube centered at $x$ with sidelength $r$ is the set $Q(x,r):=\{y\in \mathbb{R}^n:\max_{i=1,\ldots,n}\{|x_i-y_i|\}<r\}$. Given a cube $Q\subset \mathbb{R}^n$,  $\ell(Q)$ will denote its sidelength and $x_Q$ its center. For any $\lambda>0$,  $\lambda Q$ will be the sube with same center as $Q$ and sidelength $\lambda \ell(Q)$.

In the following, we will introduce some geometric notions on domains of $\mathbb{R}^n$.  First, we introduce the notion of Whitney decomposition of an open proper subset $\Omega\subset \mathbb{R}^n$, which we take from \cite[Proposition 3.3]{Diening2010} (see the references therein).
\begin{letterlemma}\label{Whitney}
There exist constants $1<c_1<c_2$ and $N>0$ such that every open subset $\Omega\subsetneq\mathbb{R}^n$ there exists a family $\{Q_j\}_{j=0}^\infty$ of cubes such that 
\begin{enumerate}
\item[(W1)] $\Omega=\bigcup_{j=0}^\infty c_1 Q_j= \bigcup_{j=0}^\infty 2c_1 Q_j$;  
\item[(W2)] $\frac{c_1}{2}\diam(Q_j)\leq d(Q_j,\partial\Omega)\leq c_2 \diam(Q_j)$;
\item[(W3)] $\sum_{j=0}^\infty \chi_{2c_1Q_j}\leq N\chi_\Omega$ on $\mathbb{R}^n$.
\end{enumerate}
Such a family is called a Whitney covering of $\Omega$ with constants $c_1, c_2$ and $N$. 
\end{letterlemma}
As it is proved in \cite{Buckley1996}, bounded John domains (which are the object of our study) and Boman chain domains are the same kind of domains. Hence, we can just focus on Boman chain domains, which we define below.

\begin{defi}\label{chain condition}
Let $\Omega$  be a domain. We say that $\Omega$ is a Boman chain domain if there exist $\sigma,N\geq 1$ such that a covering $\mathcal{W}$ of $\Omega$  with cubes can be found with the following properties:
\begin{enumerate}
\item[(B1)] $\sum_{Q\in \mathcal{W}}\chi_{\sigma Q}(x)\leq N\chi_{\Omega}(x)$, $x\in \mathbb{R}^n$;
\item[(B2)] There is a ``central cube'' $Q_0\in \mathcal{W}$ that can be connected with every cube $Q\in \mathcal{W}$ by a finite chain of cubes $Q_0,Q_1,\ldots, Q_k(Q)=Q$ from $\mathcal{W}$ such that $Q\subset NQ_j$ for $j=0,1,\ldots, k(Q)$. Moreover, $Q_j\cap Q_{j+1}$ contains a cube $R_j$ such that $Q_j\cup Q_{j+1}\subset NR_j$.
\end{enumerate}
This family $\mathcal{W}$ will be called a chain decomposition of $\Omega$ centered on $Q_0$ and with constants $\sigma$ and $N$.
\end{defi}

A fundamental fact we are going to use is that, for a John domain $\Omega$, one can build a Boman chain by using dilations of cubes in a family of Whitney cubes in such a way that these dilations still satisfy property (W2) in Lemma \ref{Whitney}. Together with this fact, we will use the following fundamental result for Boman chain domains, which allows us to obtain global inequalities for the domain from local inequalities for cubes in the chain decomposition. This result can be found in \cite{Chua1993}.
\begin{lettertheorem}\label{Chua} Let $\sigma,N\geq 1$, $1\leq q<\infty$  and $\Omega$ a Boman chain domain with chain decomposition $\mathcal{W}$ centered on a cube $Q_0$ and with constants $\sigma$ and $N$. Let $\nu$ be a measure and $w$ be a doubling weight and suppose that for each cube $Q$ in $\mathcal{W}$, we have that
\[
\|f-f_Q\|_{L^q(Q,w)} \leq A \|g\|_{L^p(\sigma Q,\nu)},
\]
with $A$ independent of $Q$. Then there exists a positive constant $C$ such that 
\[
\|f-f_{Q_0}\|_{L^q(\Omega,w)}\leq CA\|g\|_{L^p(\Omega,\nu)},
\]
where $C$ depends only on $n,q,w,\sigma$ and $N$.
\end{lettertheorem}

Now we introduce the kind of weights by means of which we are going to define our concept of ``improved'' Poincar\'e inequality. In this work we are going to consider weights which are of the form $w_{\phi }(x)=\phi(d(x))$, where $\phi$ is a positive increasing function satisfying a certain growth condition. In the fractional case, at the right hand side of the inequality, we will obtain a weight of the form $v_{\Phi,\gamma }(x,y)=\min_{z\in\{x,y\}}d(z)^{\gamma} \Phi(d(z))$, where $\Phi$ will be an appropriate power of $\phi$. For the classical case, we will obtain  a weight of the form $w_{\Phi,\gamma }(x)=d(x)^{\gamma} \Phi(d(x))$, where $\Phi$ will be an appropriate power of $\phi$.  This weights will be referred to as improving weights.

It turns out that more general objects can be written in the inequalities in Theorem \ref{Chua}. Moreover, we will take this into account together with the fact  that chains in a Boman chain domain can be taken such that they satisfy condition (W2) in Lemma \ref{Whitney} to obtain the following trivial modification of Theorem \ref{Chua}, which allows to consider weighted improved inequalities with the improving weights we just introduced above.

\begin{thm}\label{Chua2} Let $\sigma,N\geq 1$, $1\leq q<\infty$ and $\Omega$ a Boman chain domain with chain decomposition $\mathcal{W}$ centered on a cube (ball) $Q_0$ and with constants $\sigma$ and $N$. Consider two increasing functions $\phi$ and $\Phi$ with $\phi(2t)\leq c\phi(t)$. Let $\nu$ be a measure and $w$ be a doubling weight and suppose that for each cube (ball) $Q$ in $\mathcal{W}$, we have that, for some function $g$,
\[
\|f-f_Q\|_{L^q(Q,w w_\phi)} \leq A \|g\|_{L^p(\sigma Q,w_\Phi\nu )},
\]
with $A$ independent of $Q$. Then there exists a positive constant $C$ such that 
\[
\|f-f_{Q_0}\|_{L^q(\Omega,w w_\phi)}\leq CA\|g\|_{L^p(\Omega,w_\Phi\nu )},
\]
where $C$ depends only on $\mu,q,w$, $\phi$ and $\Omega$ (through the Boman and Whitney constants).
\end{thm}

We remark that the class of weights that we obtain are products of the improved weights of the form described above and weights satisfying a fractional Muckenhoupt-type condition on cubes, namely of the form
 \begin{equation}\label{condicion}
[w,v]_{A_{q,p}^{\alpha,r}(\Omega)}:=\sup_Q\ell(Q)^{\alpha}|Q|^{\frac{1}{q}-\frac{1}{p}}\left(\dashint_Qw^r\right)^{1/qr}\left(\dashint_Qv^{1-p'}\right)^{1/p'} <\infty,
\end{equation}
for some $r\geq1$ and $\alpha\in[0,1]$  where the supremum is taken over all cubes contained in a domain $\Omega\subseteq \mathbb{R}^n$.  This condition already appeared in the literature, see for instance \cite{Fefferman1983,Sawyer1992,Perez1994,Drelichman2008}

Let us denote this as $(w,v)\in A_{q,p}^{\alpha,r}(\Omega)$. This condition generalizes the classical $A_p$ condition, $p>1$, which is defined, for a weight $w\in L^1_{\loc}(\mathbb{R}^n)$ as
\[
[w]_{A_p}:=\sup_{Q}\left(\dashint_Q w\right)\left(\dashint_Q w^{1-p'}\right)^{p-1}<\infty,
\]
where the supremum is taken over all cubes $Q$ in $\mathbb{R}^n$. Observe that this coincides with $A_{p,p}^{0,1}(\mathbb{R}^n)$. As we mention in the introduction, we denote by  $A_\infty$ the  class of the weights which belong  to $A_p$ class, for some $p\geq1$.

The tool we are going to use to obtain our results is the self-improving theory developed in \cite{Franchi1998} and then sharpened in \cite{MacManus1998}. Although we will use a recently sharpened version of these results, which can be found in \cite{Perez2018}, we will introduce here the classical theory. In this way we can introduce some basic notation and get an idea of how the theory works. The concepts and results we are going to introduce below can be found in \cite{MacManus1998}.

Let us consider a space of homogeneous type $(X,d,\mu)$ and a nonnegative functional $a:\mathcal{B}\to[0,\infty)$ defined on the family $\mathcal{B}$ of all balls in $X$. Recall that $d$ denotes a quasimetric on $X$ and $\mu$ is a doubling measure with respect to $d$. The starting point of the theory is an inequality of the form 
\begin{equation}\label{starting_point}
\dashint_B|f-f_B|d\mu\leq a(B), \qquad B\in\mathcal{B},
\end{equation}
where $f$ is a locally integrable function. 

A standard instance of this situation (and the one we are going to work with) is the case where the space under study is the Euclidean space, $\mathcal{B}$ is the family of all cubes in $\mathbb{R}^n$ (and we will denote it by $\mathcal{Q}$), $f$ is a function in some suitable class of regular funcitons and $a$ is a functional of the form
\[
a(Q):=\ell(Q)^\alpha \left(\frac{\nu(Q)}{w(Q)}\right)^{1/p},\qquad Q\in\mathcal{Q},
\]
where  $\alpha>0$, $w$ is some weight (usually an $A_\infty$ weight) and  $\nu$ is some measure. Further, $\nu$  can be replaced by  the $L^p$ norm of a two-variable function  $A:\Sigma\times \mathbb{R}^n\to [0,\infty)$, where $\Sigma\subset \mathcal{Q}$. In this paper we consider, for a fixed domain $\Omega$, the function $A:\{Q\in\mathcal{Q}:Q\subset\Omega\}\times \mathbb{R}^n\to[0,\infty)$ given by  
\begin{equation}\label{def:fractional_derivative}
A(Q,y):=\int_{Q\cap B(y,\tau\ell(Q))}\frac{|f(x)-f(y)|^p}{|x-y|^{n+sp}}dx\chi_Q(y),\qquad 0<s<1
\end{equation}
which, in the case of cubes with sidelength proportional to the distance to the boundary, can be bounded (up to a reparametrization on $\tau$)    by the function  $B:\mathbb{R}^n\to[0,\infty)$ given by
\begin{equation}
B(y):=\int_{\Omega\cap B(y,\tau d(y))}\frac{|f(x)-f(y)|^p}{|x-y|^{n+sp}}dx\chi_\Omega(y),\qquad 0<s<1
\end{equation}
We will write $A(Q,y):= |\nabla_{s,p,Q}^\tau f|(y)$ and $B(y):= |\nabla_{s,p,\Omega }^\tau f|(y)$ and we will call them fractional (resp. local fractional) derivatives of $f$ at the point $y\in \Omega$.

The classical self-improving theory allows to obtain, under some geometric  conditions on $a$ with respect to an $A_\infty$ weight $w$, an improvement of \eqref{starting_point} of the form 
\[
\|f-f_B\|_{L^{r,\infty}\left(B,\frac{w}{w(B)}\right)}\leq C a(B),\quad B\in \mathcal{B},
\]
where
\[
\|f-f_B\|_{L^{r,\infty}\left(B,\frac{w}{w(B)}\right)}:=\sup_{t>0}t\left(\frac{w(\{x\in B:|f(x)-f_B|^r>t\}}{w(B)}\right)^{1/r}.
\]

The precise geometric condition on the functional $a$ (and the weight $w$) is the following one.
\begin{defi}\label{Dr}
Let $0<r<\infty$ and $w$ be a weight. We say that the functional $a$ satisfies the weighted $D_r(w)$ condition if there exists a finite constant $C>0$ such that for each ball $B$ and any family $\{B_i\}$ of pairwise disjoint sub-balls of  $B$,
\[
\sum_{i}a(B_i)^rw(B_i)\leq C^r a(B)^r w(B).
\]
The best constant $C$ for which this happens will be denoted by $\|a\|$. 
\end{defi}
Observe that by definition, $\|a\|\geq1$. 

This condition was used in \cite{MacManus1998} to prove the following result.  $K$ will denote the constant such that the pseudometric $d$ satisfies $d(x,z)\leq K(d(x,y)+d(y,z))$ for points $x,y,z\in X$.

\begin{lettertheorem}\label{weakimproving}
Let $B_0$ be a ball in $X$ and let $\delta>0$. Set $\widehat{B_0}=(1+\delta)KB_0$. Suppose that the functional $a$ satisfies the weighted $D_r(w)$ condition of Definition \ref{Dr} for some $0<r<\infty$ and some $w\in A_\infty(\mu)$. If $f$ is a locally integrable function on $\widehat{B_0}$ for which there exists a constant $\tau\geq 1$  such that for all balls $B$ with $B\subset \widehat{B_0}$
\[
\dashint_B |f-f_B|d\mu \leq a(\tau B),
\]
then there exists a constant $C$ independent of $f$ and $B_0$ such that
\[
\|f-f_{B_0}\|_{L^{r,\infty}\left(B_0,\frac{w}{w(B_0)}\right)}\leq C\|a\|a(\tau \widehat{B_0}).
\]
\end{lettertheorem}

We remark that this result implies, by Kolmogorov's inequality, a strong inequality for any  $p<r$. It is also known that for some special functionals $a$ it is possible to obtain  the corresponding strong inequality with exponent $r$ from the weak one through the so called truncation method. Indeed, actually the examples of functionals we are going to consider are among these functionals (see for instance \cite{Hajlasz2000} and \cite{Dyda2016}, where the so-called Maz'ya's truncation method or the ``weak implies strong'' argument is used).

Recently, in \cite{Perez2018}, the authors have proved a better self-improvement in the Euclidean case given that the functional $a$ satisfies a stronger geometric condition, which we state below. First we have to give the notion of smallnes of a disjoint family of subcubes of a given cube $Q$.
\begin{defi}
Let $L>1$ and let $Q$ be a cube. We will say that a family of pairwise disjoint subcubes $\{Q_i\}$ of $Q$ is $L$-small if
\[
\frac{\left|\bigcup_{i}Q_i\right|}{|Q|}\leq \frac{1}{L}.
\]
We denote this by $\{Q_i\}\in S(L,Q)$.
\end{defi}
 The modified notion of $D_r$-type condition is the following one.
 \begin{defi}\label{SDr}
 Let $w$ be a weight and $s>1$. We say that the functional $a$ satisfies the weighted $SD^s_r(w)$ condition $0\leq r<\infty$ if there is a finite constant $C>0$ such that for any cube $Q$ and any family $\{Q_i\}\in S(L,Q)$, $L>1$, the following inequality holds:
 \[
 \sum_ia(Q_i)^rw(Q_i)\leq C^r\left(\frac{1}{L}\right)^{r/s}a(Q)^rw(Q).
 \]
 We write in this case $a\in SD^s_r(w)$ and we say that $a$ preserves the smallness condition of the family of cubes. As before, the smallest $C$ for which this happens will be denoted by $\|a\|$. 
 \end{defi}

Observe that now $\|a\|$ does not need to be larger than $1$.

Under these conditions, the authors prove the following
result.
\begin{lettertheorem}\label{PRimprovement}
Let $w\in A_\infty$. Consider also the functional $a$ such that for some $p\geq 1$ it satisfies the weighted condition $SD_p^s(w)$ with $s>1$ and constant $\|a\|$. Let $f$ be a locally integrable function such that
\[
\dashint_Q |f-f_Q|\leq a(Q),
\]
for every cube $Q$. Then, there exists a dimensional constant $C_n$ such that, for any cube $Q$
\begin{equation}\label{mejora}
\left(\dashint_Q|f-f_Q|^pdw\right)^{1/p}\leq C_ns\|a\|^sa(Q).
\end{equation}
\end{lettertheorem}

The authors remark that the $A_\infty$ condition should be avoidable, as the $A_\infty$ constant does not appear in \eqref{mejora}. Actually, the $A_\infty$  condition is just used to prove that some a priori quantity is finite. The main theorem of the present paper avoids the artifice   of the $A_\infty$ condition in the proof by another one, which turns out to give a slightly more general result thanks to the reverse H\"older condition of $A_\infty$ weights. The result we will obtain will give  us a way to obtain  improved Poincar\'e inequalities with weights without using  the truncation method.

In the following, we give some results which can be obtained  as a byproduct of our generalization of Theorem \ref{PRimprovement}. We  start with the weighted improved Poincar\'e inequalities obtained in \cite{Drelichman2008}. The main result the authors prove there is, with our notation, the following.

\begin{lettertheorem}
Let $\Omega\subset\mathbb{R}^n$ be a bounded John domain and let $1<p<q<\infty$. If $(w,v)\in A_{q,p}^{1-\alpha,1}(\mathbb{R}^n)$ and $w,v^{1-p'}$ are reverse doubling weights, then
\begin{equation}\label{DD}
\inf_{a\in\mathbb{R}}\|f-a\|_{L^q(\Omega,w)}\leq C \||\nabla f|d^\alpha\|_{L^p(\Omega,v)},
\end{equation}
for all locally Lipschitz $f\in L^q(\Omega,w)$. If  $p=q$, then the result is obtained for weights $w$ and $v$ such that  $w,v^{1-p'}$ are reverse doubling weights and
\begin{equation}\label{DDp=q}
\sup_Q\ell(Q)^{\alpha}|Q|^{\frac{1}{q}-\frac{1}{p}}\left(\dashint_Qw^r\right)^{1/qr}\left(\dashint_Qv^{(1-p')r}\right)^{1/p'r} <\infty,
\end{equation}
 for some $r>1$.
\end{lettertheorem}
As the authors remark, here we may assume $q\leq \frac{np}{n-p(1-\alpha)}$, since otherwise $w$ equals zero almost everywhere on $\{v<\infty\}$, as it was observed in \cite[Remark b]{Sawyer1992}.

We are able to obtain, in Theorem \ref{unified}, inequality \eqref{DD} under the assumptions $(w,v)\in A_{q,p}^{1-\alpha,r}(\Omega)$, $w$  doubling and $r>1$. Note that no extra assumptions are needed in $v$ and also that $A_{p,p}^{1-\alpha,r}(\Omega)$ is weaker than \eqref{DDp=q}.  Also, our result allows us to obtain a $(w_\phi w,w_{\Phi,\alpha}v)$-weighted improved version, where $\Phi(t)=\phi(t)^{\frac{p}{q}}$, $\alpha\in [0,1]$ and $(w,v)\in A_{q,p}^{1-\alpha,r}(\Omega)$, with $w$ a doubling weight. 

The second result we are going to focus on is  the recent improved fractional Poincar\'e-Sobolev inequality obtained in \cite{Cejas2019} in the general context of Ahlfors-David regular metric spaces. In the context of the Euclidean space, the result they obtain reads as follows.
\begin{lettertheorem}\label{Euclidean case2}
Let $\Omega$ in $\mathbb{R}^n$ be a bounded John domain and $1<p<\infty$. Given the parameters $s,\tau\in(0,1)$, $0\leq \gamma <s$ such that $(s-\gamma)p<n$ and $\phi$ an increasing function with $\phi(2t)\leq \phi(t)$ and such that $w_\phi\in L^1_{loc}(\Omega)$, if we define $q=\frac{np}{n-(s-\gamma)p}$, the inequality 
\[
\begin{split}
\inf_{c\in \mathbb{R}}\|u-c&\|_{L^{q}(\Omega,w_\phi)}\lesssim\left(\int_\Omega\int_{\{z\in \Omega:|z-y|\leq \tau d(y)\}}\frac{|u(z)-u(y)|^p}{|z-y|^{n+sp}}v_{\Phi,\gamma}(z,y)dzdy\right)^{\frac{1}{p}},
\end{split}
\]
holds for any function $u\in W^{s,p}(\Omega,dx)$, where $\Phi(t)=\phi(t)^{\frac{p}{q}}$.
\end{lettertheorem} 
Their result is based on an appropriate representation formula, duality and the boundedness of the Riesz potential. Our approach avoids any of these facts and in particular avoids any representation formula. Also, we are able to obtain the corresponding $(w_\phi w,w_{\Phi,\gamma}v)$-weighted version of the inequality, where $(w,v)\in A_{q,p}^{s-\gamma,r}(\Omega)$ and $w$ is a doubling weight. Thus, we improve the results in \cite{Cejas2019}  for the special case where $X$ is the Euclidean space and $F$ is equal to $\partial\Omega$.

The fundamental idea for obtaining our results is to obtain a suitable starting point to use the self-improving theory. Then, by applying a more general version of Theorem \ref{PRimprovement}, we obtain an improvement of the starting point on cubes of the domain, and so, by concatenating these self-improvements on Whitney cubes of  Boman chains of a John domain by means of Theorem \ref{Chua2}, we can obtain the weighted improved Poincar\'e inequality on the whole domain. 

\section{Main theorem: self-improvements without the $A_\infty$ assumption}\label{sec:improvement}

As we mentioned above, apparently the $A_\infty$ condition is not actually   needed for obtaining the self-improvement  result in Theorem \ref{PRimprovement}. Actually, the only reason why the $A_\infty$ condition is used in their argument is in order to get that some intermediate quantity considered in the proof is finite. Once this is obtained, the rest of the argument does not need this condition on the weight. The precise property which is needed from the weight is the fact that, given an $A_\infty$ weight $w$, one has the existence of constants $C,\delta>0$ such that
\begin{equation}\label{Ainfty}
\frac{w(E)}{w(Q)}\leq C\left(\frac{|E|}{|Q|}\right)^{\delta}
\end{equation}
holds for any cube $Q$ and any measurable subset $E\subset Q$.  Actually, this is the original definition of the $A_\infty$ class of weights which turns out to be equivalent to saying that the Fujii-Wilson's  quantity \eqref{Ainfty_constant} is finite.

Thanks to this, one is able to prove that, for a given functional $a$ satisfying the smallness condition $SD_p^s(w)$ in Definition \ref{SDr}, the perturbation $a_\varepsilon$ of $a$, defined as $a_\varepsilon(Q):=a(Q)+\varepsilon$ for any cube $Q$ and any $\varepsilon>0$, also satisfies a smallnes condition $SD_p^{\tilde{s}}(w)$,  where $\tilde{s}$ is a constant bigger than $s$ and which depends on the $A_\infty$ constant of $w$. This allows to prove that the quantity we referred to above is finite independently of the value of $\varepsilon$, and then, by taking limit when $\varepsilon$ goes to $0$, one is able to obtain the desired result in Theorem \ref{PRimprovement}.

What we will do is to use the same idea but replacing the weight $w$ with a functional defined by using the weight  and satisfying a condition like \eqref{Ainfty}.  Thus, we are going to be able to perform the same argument for any functional satisfying a smallness condition with respect to this new functional, and then we will obtain a result in the spirit of Theorem \ref{PRimprovement} without assuming the $A_\infty$ property on the weight.

To be precise, we will work, for an $r>1$ and a weight $w\in L^1_{\loc}(\mathbb{R}^n)$, with the functional $w_r$ given by the formula
\[
w_r(Q)=|Q|\left(\dashint_Q w^r\right)^{1/r}.
\]
This kind of functionals already appeared in some works as for instance   \cite{Perez1995,Cruz-Uribe2000,Cruz-uribe2002}, in which the authors study sufficient conditions for the two-weighted weak and strong-type (respectively) boundedness of fractional integrals, Calder\'on-Zygmund operators and commutators. There, one can find the following straightforward properties of $w_r$: 
\begin{enumerate}
\item $w(E)\leq w_r(E)$ for any measurable nonzero measure set $E$.
\item If $E\subset F$ are two nonzero measure sets, then 
\begin{equation}\label{w_r1}
w_r(E)\leq \left(\frac{|E|}{|F|}\right)^{1/r'}w_r(F).
\end{equation}
\item If $E=\bigcup_i E_i$ for some disjoint family $\{E_i\}_i$, then 
\begin{equation}\label{w_r2}
\sum_i w_r(E_i) \leq w_r(E).
\end{equation}
\end{enumerate}
Condition \eqref{w_r1} is what will allow us to work with perturbations of a functional $a$ without assuming the $A_\infty$ condition on the weight. However, we have to ask $a$ to satisfy an adapted $SD_p^s$ condition. 
\begin{defi}\label{def:SDp}
Let  $s>1$ and $r\geq1$ and let $w$ be any   weight. We say that the functional $a$ satisfies the weighted $SD^s_p(w_r)$ condition for $0\le p<\infty$ if there is a constant $C$ such that for any cube (or ball) $Q$ and any family $\{Q_i\}$ of pairwise disjoint subcubes (resp. subballs) of $Q$ such that $\{Q_i\}\in S(L,Q)$, the following inequality holds:
\begin{equation}\label{eq:SDp}
\sum_{i}a(Q_i)^pw_r(Q_i)\le C^p \left (\frac{1}{L}\right )^{\frac{p}{s}}a(Q)^pw_r(Q).
\end{equation}
The best possible constant $C$ above is denoted by $\|a\|$ and also we will write in this case that $a\in SD^s_p(w_r)$.  
\end{defi}

One should note that, for $A_\infty$ weights, $SD_p^s(w)$ and $SD_p^s(w_r)$ conditions are equivalent if we take $r\in[1,1+\varepsilon_w]$, where $\varepsilon_w>0$ depends on the $A_\infty$ condition on $w$. This comes from the fact that every $A_\infty$ weight is in a reverse H\"older class $\RH_{r_w}$, $r_w>1$, and then one has that $w(Q)\asymp w_r(Q)$   for every cube $Q$ and any $r\in[1,r_w]$.  Hence our result contains the main result in \cite{Perez2018}.

The result we obtain is the following one. 
\begin{thm}\label{thm:Lp(w)-a(Q)-clean-sht-euclidean}

Let $w$ be any weight on $\mathbb{R}^n$. Consider also a functional $a\in SD^s_q(w_r)$ with $s,r>1$, $q>0$ and constant $\|a\|$. Let $f$ be a locally integrable function such that
\begin{equation}\label{startingpoint}
\frac{1}{|Q|}\int_{Q} |f-f_{Q}|dx \le a(Q),
\end{equation}
for every cube $Q$.
Then, there exists a dimensional constant $C_{n}$ such that for any  cube $Q$%
\begin{equation}\label{eq:First main estimate}
\left( \frac{1}{ w_r(Q)  } \int_{ Q }   |f -f_{Q}|^q     \,w\right)^{\frac{1}{q}}   \leq  C_{n}\, s \|a\|^s  a(Q).
\end{equation}
\end{thm}

One should also observe that what we obtain is not a bound for the $L^q(w)$ average of the oscillation of $f$ over $Q$, as we get $w_r(Q)$ instead of $w(Q)$ in the denominator.  However,  we emphasize that the result holds for any weight $w$.  The fact that we do not get an $L^q(w)$ average will not   cause any problem for our applications as, for the functionals $a$ we are going to consider,  the appearance of the quantity $w_r(Q)$  balances the condition. We will see the details of this in the following section and the rest of this section will be devoted to the proof of Theorem \ref{thm:Lp(w)-a(Q)-clean-sht-euclidean}. 
We also emphasize the fact that the case $q$ below $1$ is also included,  and the same can be done in Theorem \ref{PRimprovement}.

\begin{proof}[Proof of Theorem \ref{thm:Lp(w)-a(Q)-clean-sht-euclidean}]
Take a cube $Q$. By the hypothesis, we know that
\[\dashint_Q|f-f_Q|dx\leq a(Q),\]
which can be reformulated as 
\[\dashint_Q\frac{|f-f_Q|}{a(Q)}dx\leq 1.\]
We consider $L>1$ to be chosen later.  Let us perform the standard local Calder\'on-Zygmund decomposition (see \cite{Calderon1952}) for $\frac{f-f_Q}{a(Q)}$ at level $L$. This yields a collection $\{Q_j\}_{j\in\mathbb{N}}$ in the family $\mathcal{D}(Q)$ of dyadic subcubes of $Q$, maximal with respect to inclusion, satisfying 
\[
L\leq \dashint_{Q_j}\frac{|f-f_Q|}{a(Q)}dx\leq 2^n L,
\]
where the second inequality follows by maximality. As in \cite{Perez2018} note that
\[
\Omega_L:=\left\{x\in Q:M_Q^d\left(\frac{f-f_Q}{a(Q)}\chi_Q\right)(x)>L\right\}=\mathring\bigcup_{j\in\mathbb{N}} Q_j,
\]
where $M_Q^d$ is the localized dyadic maximal function asociated to the cube $Q$, \emph{i.e.}
\[
M_Q^d g(x):=\sup_{\underset{P\in \mathcal{D}(Q)}{x\in P\subset Q}}\dashint_P|g|dx.
\]
Then, by the Lebesgue differentiation theorem it follows that
\begin{equation}\label{LDT}
\frac{|f(x)-f_Q|}{a(Q)}\leq L,\qquad\text{a.e. }x\notin\Omega_L.
\end{equation}

We now decompose $\frac{f-f_Q}{a(Q)}$ as follows
\begin{equation*}\label{decomposition}
\begin{aligned}
\frac{f-f_Q}{a(Q)}&= \frac{f-f_Q}{a(Q)}\chi_{Q\backslash\Omega_L}+\frac{f-f_Q}{a(Q)}\chi_{\Omega_L}=\frac{f-f_Q}{a(Q)}\chi_{Q\backslash\Omega_L}+\sum_{j\in\mathbb{N}}\frac{f-f_Q}{a(Q)}\chi_{Q_j}\\
&=\frac{f-f_Q}{a(Q)}\chi_{Q\backslash\Omega_L}+\sum_{j\in\mathbb{N}}\frac{f-f_{Q_j}}{a(Q)}\chi_{Q_j}+
\sum_{j\in\mathbb{N}}\frac{f_{Q_j} -f_Q}{a(Q)}\chi_{Q_j}
\\
&=\left[\frac{f-f_Q}{a(Q)}\chi_{Q\backslash\Omega_L}+\sum_{j\in\mathbb{N}}\frac{ f_{Q_j} -f_Q}{a(Q)}\chi_{Q_j}\right]+
\sum_{j\in\mathbb{N}}\frac{f-f_{Q_j} }{a(Q)}\chi_{Q_j},
\end{aligned}
\end{equation*}
 and note that, by the properties of the cubes $Q_j$ and \eqref{LDT}, we have that the bracket above is bounded by $2^nL$.

Now we start with the estimation of the desired $L^q$ norm.  Let us take first $q\geq 1$. Consider on $Q$ the measure $\nu$ defined by $d\nu=\frac{w\chi_Q dx}{w_r(Q)}$. Then, by the triangle inequality, we get
\begin{equation}\label{estimacion_norma}
\left(\frac{1}{w_r(Q)}\int_Q \frac{|f-f_Q|^q}{a(Q)^q}wdx\right)^{1/q} \leq   2^nL+\left(\frac{1}{w_r(Q)}\int_{\Omega_L}\left|\sum_{j\in\mathbb{N}}\frac{f-f_{Q_j} }{a(Q)}\chi_{Q_j}\right|^qwdx\right)^{1/q},
\end{equation}
where we used that $w(Q)\leq w_r(Q)$ and also the bound for the bracket in the above decomposition of $\frac{f-f_Q}{a(Q)}$. Now observe that thanks to the disjointness of the cubes in the Calder\'on-Zygmund decomposition, we can plug the power $q$ inside the sum in the second term above. Thus, this term can be bounded as follows
\[
\begin{split}
\int_{\Omega_L}\left|\sum_{j\in\mathbb{N}}\frac{f-f_{Q_j} }{a(Q)}\chi_{Q_j}\right|^qwdx&= \sum_{j\in\mathbb{N}}\int_{Q_j}\left|\frac{f-f_{Q_j} }{a(Q)} \right|^qwdx\\
&=\frac{1}{a(Q)^q}\sum_{j\in\mathbb{N}}\frac{a(Q_j)^qw_r(Q_j)}{w_r(Q_j)}\int_{Q_j}\left|\frac{f-f_{Q_j}}{a(Q_j)}\right|^qwdx\\
&\leq \frac{X^q}{a(Q)^q}\sum_{j\in\mathbb{N}}a(Q_j)^qw_r(Q_j),
\end{split}
\]
where $X$ is the quantity defined by
\begin{equation}\label{eq:a priori estimate}
X:=\sup_{P }\left(\frac{1}{w_r(P)}\int_P\left|\frac{f-f_P}{a(P)}\right|^qwdx\right)^{1/q},
\end{equation}
where the supremum is taken among all cubes in $\mathbb{R}^n$  which we assume is finite for the time being.

By the defining property for the selected cubes $Q_j$, $j\in\mathbb{N}$, we know that $\{Q_j\}_{j\in\mathbb{N}}\in S(L,Q)$, since 
\[
\begin{split}
\sum_{j\in\mathbb{N}} |Q_j|&\leq \frac{1}{L}\sum_{j\in\mathbb{N}}\int_{Q_j}\frac{|f-f_Q|}{a(Q)}dx=\frac{1}{L}\int_{\mathring\bigcup_{j\in\mathbb{N}}Q_j}\frac{|f-f_Q|}{a(Q)}dx\\
&\leq \frac{|Q|}{L}\dashint_Q\frac{|f-f_Q|}{a(Q)}dx\leq \frac{|Q|}{L}.
\end{split}
\]

Then, by the $SD_r^s(w)$ condition, we obtain that
\[
\begin{split}
\left(\frac{1}{w_r(Q)}\int_Q \frac{|f-f_Q|^q}{a(Q)^q}wdx\right)^{1/q}&\leq 2^nL+X\left(\frac{\sum_{j\in\mathbb{N}}a(Q_j)^qw_r(Q_j)}{a(Q)^qw_r(Q)}\right)^{1/q}\\
&\leq 2^nL+X\frac{\|a\|}{L^{1/s}}.
\end{split}
\]
Now we take supremum on the left-hand side and choose $L=2e\max\{\|a\|^s,1\}$, so, by an absortion argument, the above inequality becomes
\[
X\leq 2^{n+1}e\|a\|^s\left((2e)^{1/s}\right)'\leq 2^{n+1}e\|a\|^ss.
\]
This  implies  the desired inequality.

Observe that the same argument can be performed in the case $q<1$ by considering the $L^q$ norm to the power $q$ in \eqref{estimacion_norma}  and using the triangle inequality inside the integral.

To perform the absortion argument above, we need $X$ to be finite, but this  is ensured by the properties $w_r(E)\leq \left(\frac{|E|}{|Q|}\right)^{1/r'}w_r(Q)$ and $\sum_{j}w_r(Q_j)\leq w_r(Q)$ whenever $E\subset Q$ and $\bigcup_j Q_j=Q$, the sets $Q_j$ being pairwise disjoint. These properties are used along with a perturbation argument which leds us to work with $a_\varepsilon(Q)=a(Q)+\varepsilon$ instead of with $a(Q)$.

Indeed, consider a cube $Q$ and a family $\{Q_j\}\in S(L,Q)$. Then there exist constants $C$ and $\tilde{s}$ larger than $\|a\|$ and $s$ such that
\[\sum_ja_\varepsilon(Q_j)^qw_r(Q_j)\leq C^q\left(\frac{1}{L}\right)^{q/\tilde{s}}a_\varepsilon(Q)^qw_r(Q).  \]
This follows from   the Minkowiski's inequality in the case $q\geq 1$:
\[
\begin{split}\left(\sum_ja_\varepsilon(Q_j)^qw_r(Q_j)\right)^{1/q}&=\left(\sum_j(a(Q_j)+\varepsilon)^qw_r(Q_j)\right)^{1/q}\\
&\leq  \left( \sum_ja(Q_j)^qw_r(Q_j)\right)^{1/q}+\left( \sum_j\varepsilon^qw_r(Q_j)\right)^{1/q}\\
& \leq \frac{\|a\|}{L^{1/s}}a(Q)w_r(Q)^{1/q}+\varepsilon w_r\left(\bigcup_jQ_j\right)^{1/q},
\end{split}
 \]
where we should write   $2^{\frac1q-1}$ as a factor in the second and third lines whenever $q<1$. 
 
Now we use the properties of $w_r$   to obtain
\[
 w_r\left(\bigcup_jQ_j\right)\leq \left(\frac{|\bigcup_j Q_j|}{|Q|}\right)^{1/r'}w_r(Q)\leq \frac{1}{L^{1/r'}}w_r(Q),
\]
since $\{Q_j\}\in S(L,Q)$.

Thus,   if $q>1$  
\[
\begin{split}\left(\sum_ja_\varepsilon(Q_j)^qw_r(Q_j)\right)^{1/q}& \leq \frac{\|a\|}{L^{1/s}}a(Q)w_r(Q)^{1/q}+\frac{\varepsilon}{L^{1/qr'}} w_r\left(Q\right)^{1/q}\\
&\leq \max\left\{ \frac{\|a\|}{L^{1/s}},\frac{\varepsilon}{L^{1/qr'}} \right\}a_\varepsilon(Q)w_r(Q)^{1/q}\\
&\leq \frac{\max\{\|a\|,1\}}{L^{1/\max\{s,qr'\}}}a_\varepsilon(Q)w_r(Q)^{1/q},
\end{split}
 \]
where again a factor $2^{\frac1q-1}$ should be added in case $q<1$. This just affects to the quantity $\|a\|$ which now becomes  $C=\max\{\|a\|,1\}$ (resp. $C=2^{\frac1q-1}\max\{\|a\|,1\}$ if $q<1$). The new $ s$ is  $\tilde{s}=\max\{s,qr'\}$. 
 
Hence, with the same proof as before, and taking into account that 
\[X_\varepsilon =\sup_{P}\left(\frac{1}{w_r(P)}\int_P\left|\frac{f-f_P}{a_\varepsilon(P)}\right|^qwdx\right)^{1/q}\leq \frac{\|f-f_Q\|_{L^\infty}}{\varepsilon}<\infty,\]
we can run the argument by choosing $L$ large enough independent of $\varepsilon$ to obtain 
$$
X_{\varepsilon}\leq  c_{s,q,r',\|a\|}
$$
for any $\varepsilon$. This yields the finiteness of \eqref{eq:a priori estimate}, which ends the proof of the theorem.

\end{proof}

\section{Some applications of the self-improving result}\label{sec:applications}

This section will be devoted to the obtention of weighted improved Poincar\'e type inequalities as the ones describes in the first two sections.  More precisely, the result we obtain is the following one.
\begin{thm}\label{unified}
Let $s\in(0,1]$ and $0\leq \gamma\leq s$. Consider $1<p\leq q\leq \frac{np}{n-(s-\gamma)p}$. Let $\Omega$ be  a bounded John domain and consider an  increasing function $\phi$  with $\phi(2t)\leq C\phi(t)$ such that $w_\phi\in L^1_{\loc}(\Omega)$.  Let $w$ a doubling weight and $v$ a weight. If $f\in W_{\tau}^{s,p}\left(\Omega\right)$ for $\tau\in(0,1)$ and $(w,v)\in A_{q,p}^{s-\gamma,r}$ for some $r>1$, then 
\[
\inf_{c\in\mathbb{R}}\|f-c\|_{L^q(\Omega, w_\phi w)}\lesssim[f]_{W_{\tau}^{s,p}\left(\Omega,w_{\Phi,\gamma p}v\right)}.
\]
When $s<1$, the right hand side of the inequality above can be replaced by the quantity  $[f]_{W_{\tau}^{s,p}\left(\Omega,v_{\Phi,\gamma p}v \right)}$.

\end{thm}
\begin{proof}
Our results follows from the result in Section \ref{sec:improvement} and the following observation. Let us consider a measure $\nu$, a number $\alpha\in[0,1]$ and a weight $w$ and let us define the functional $a(Q):=\ell(Q)^\alpha\left(\frac{\nu(Q)^{1/p}}{w_t(Q)^{1/q}}\right)$ for any cube $Q$, where $t\geq1$. This functional satisfies the $SD_q^s(w_r)$ for any $t\geq r\geq 1$ with $s=n/\alpha q$.
Indeed, take $\{Q_j\}\in S(L)$ a family of subcubes of a cube $Q$. Then by using H\"older's inequality for $w_t$  and also for the sum with exponent $\frac{n}{q\alpha}>1$ (in case $q\alpha/n\geq1$ the result is immediate)
\[
\begin{split}
\sum_ja(Q)^qw_r(Q)&\leq \sum_j |Q_j|^{\frac{q\alpha}{n}}\nu(Q_j)^{q/p}\leq \left(\sum_j|Q_j|\right)^{\frac{q\alpha}{n}}\left(\sum_j\nu(Q_j)^{\frac{(n/q\alpha)'q}{p}}\right)^{\frac{1}{(n/q\alpha)'}}\\
&\leq \left(\frac{1}{L}\right)^{\frac{q\alpha}{n}}\ell(Q)^{q\alpha }\nu(Q)^{q/p}\leq\left(\frac{1}{L}\right)^{\frac{q\alpha}{n}}a(Q)^qw_r(Q).
\end{split}
\]

In what follows, we are going to get a starting point \eqref{startingpoint} where $a(Q)$ is of the form given above and where $\nu$ and $w$ are defined by means of the weights $w$ and $v$ in condition \eqref{condicion}.

We start by noting that, for any function $f\in W^{1,1} (\Omega)$, we have that, if $Q$ is a cube in $\Omega$, then, by the classical Poincar\'e inequality,
\begin{equation}\label{startingclassical}
\begin{aligned}
\dashint_Q|f(x)-f_Q|dx&\leq C \ell(Q)\dashint_Q|\nabla f(x)|dx =C\ell(Q)^{1-\gamma}\ell(Q)^\gamma \dashint_Q |\nabla f(x)|dx  \\
&\leq C\ell(Q)^{1-\gamma} \left(\dashint_Q v^{1-p'}\right)^{\frac{1}{p'}}\ell(Q)^\gamma\left(\dashint_Q|\nabla f(x)|^pv(x)dx\right)^{\frac{1}{p}}\\
& \leq C[w,v]_{A_{q,p}^{1-\gamma,r}(\Omega)}\ell(Q)^\gamma\left(\frac{1}{w_r(Q)^{\frac{p}{q}}}\int_Q|\nabla f(x)|^pv(x)dx\right)^{\frac{1}{p}},
\end{aligned}
\end{equation}
where we have used H\"older's  inequality and the $A_{q,p}^{1-\gamma,r}(\Omega)$ condition on $w$ and $v$.

Then, we have obtained \eqref{startingpoint}  with the special functional
$$a_{1,p}(Q):= C[w,v]_{A_{q,p}^{1-\gamma,r}(\Omega)}\ell(Q)^\gamma\left(\frac{1}{w_r(Q)^{\frac{p}{q}}}\int_Q|\nabla f(x)|^pv(x)dx\right)^{\frac{1}{p}}$$
for any function $f\in W^{1,1}(\Omega)$.  This functional will satisfy the smallnes condition $SD_q^{\frac{n}{\gamma q}}(w_r)$ condition as long as $|\nabla f|\in L_{\loc}^p(\Omega,v)$. On the other hand, if it does not satisfy this condition, then there is nothing to prove, as the right-hand side of the inequalities under consideration will be infinite. This starting point will alow us to obtain a weighted improved classical Poincar\'e-Sobolev inequality.

Now, we will get an starting point \eqref{startingpoint} which allows us to obtain a weighted improved fractional Poincar\'e-Sobolev inequality. Once we get this starting point, we will be able to obtain our main result with a unified approach by applying the self-improving result we proved in Section \ref{sec:improvement} and a modified version of the standard chaining argument which we stated in Theorem \ref{Chua2}.

Let us consider a sufficiently regular function $f$ so that the following computations make sense.  We will be using the following construction which  can be found in \cite{Hurri-Syrjaenen2013}. 

\begin{letterlemma}\label{HV}
For any cube $Q$ in a domain  $\Omega\subset\mathbb{R}^n$ and $0<\tau<1$, we can define a family $\mathcal{Q}$ of subcubes of $Q$ with the following properties:
\begin{enumerate}
\item The size of every cube in $\mathcal{Q}$ is comparable to that of $Q$.
\item If $Q_1,Q_2\in\mathcal{Q}$ share a common face, then the set $R=Q_1\cup Q_2$ is a set  of size comparable to that of $Q$ which satisfies that $R\subset B(y,\tau \ell(Q))$ for every $y\in R$.
\end{enumerate}
Observe that, given $x,y\in R$, one has $B(y,d(x,y))\subset CQ_1\cup CQ_2$ for some $C\geq1$. This family of subcubes is uniformly finite for every cube $Q$.
\end{letterlemma}

For any of such sets $R$ one has, by convexity, the following:  \begin{equation}\label{starting1}
\begin{aligned}
\dashint_R|f-f_R|&\leq \dashint_R\dashint_R|f(x)-f(y)|dxdy \leq \dashint_R\dashint_R|f(x)-f(y)| dx\,v(y)^{\frac1p-\frac1p}dy\\
&\leq \left(\dashint_R\dashint_R|f(x)-f(y)|^pdx\, v(y)dy\right)^{1/p}\left(\dashint_Rv^{1-p'}\right)^{1/p'}\\
&\leq \left(\dashint_R\int_{R}\frac{|f(x)-f(y)|^p}{|x-y|^{n}}dx\, v(y)dy\right)^{1/p}\left(\dashint_Rv^{1-p'}\right)^{1/p'}\\
&\leq \ell(Q)^\gamma \ell(Q)^{s-\gamma}\left(\dashint_R     |\nabla_{s,p,Q}^\tau f|(y)^p       v(y)dy\right)^{1/p}\left(\dashint_Rv^{1-p'}\right)^{1/p'} \\
&\asymp\frac{\ell(Q)^\gamma \ell(Q)^{s-\gamma}}{|Q|^{1/p}}\left(\int_R     |\nabla_{s,p,Q}^\tau f|(y)^p       v(y)dy\right)^{1/p}\left(\dashint_Rv^{1-p'}\right)^{1/p'}\\
&\leq\frac{[w,v]_{A_{q,p}^{s-\gamma,r}(\Omega)}\ell(Q)^\gamma}{w_r(Q)^\frac{1}{q}}\left(\int_Q   |\nabla_{s,p,Q}^\tau f|(y)^p       v(y)dy\right)^{1/p},
\end{aligned}
\end{equation}
where we have assumed $(w,v)\in A_{q,p}^{s-\gamma,r}(\Omega)$ and $|\nabla_{s,p,Q}^\tau f|$ is the function already defined in \eqref{def:fractional_derivative} by
\begin{equation*}
|\nabla_{s,p,Q}^\tau f|:=\int_{Q\cap B(y,\tau\ell(Q))}\frac{|f(x)-f(y)|^p}{|x-y|^{n+sp}}dx\chi_Q,\qquad 0<s<1.
\end{equation*}

Summarizing, we obtained
\begin{equation}\label{starting2}
\begin{aligned}
\dashint_R&|f-f_R|\leq \frac{[w,v]_{A_{q,p}^{s-\gamma,r}(\Omega)}\ell(Q)^\gamma}{|Q|^{1/q}\left(\dashint_Qw^r\right)^{1/qr}}\left(\int_Q   |\nabla_{s,p,Q}^\tau f|(y)^p       v(y)dy\right)^{1/p}.
\end{aligned}
\end{equation}

We just have to argue as in \cite[Lemma 2.2]{Hurri-Syrjaenen2013} in order to get (by the above and the  doubling metric property of $\mathbb{R}^n$)  that for any cube $Q\subset \Omega$
\begin{equation}\label{startingfractional}
\begin{aligned}
\dashint_Q&|f-f_Q| \leq \frac{C_n[w,v]_{A_{q,p}^{s-\gamma,r}(\Omega)}\ell(Q)^\gamma}{|Q|^{1/q}\left(\dashint_Qw^r\right)^{1/qr}}\left(\int_Q   |\nabla_{s,p,Q}^\tau f|(y)^p       v(y)dy\right)^{1/p}.
\end{aligned}
\end{equation}
Observe that the right-hand side defines, for  cubes $Q\subset \Omega$, a functional of the form $a_s(Q)=\ell(Q)^\alpha\frac{\nu(Q)^{1/p}}{w_r(Q)^{1/q}}$ with the weight $w$, $\alpha=\gamma$  and $\nu(Q)=\int_Q|\nabla_{s,p,Q}^\tau f|(y)^p v(y)$. 

The assumptions we need on $f$ are those which ensure the $L^p(Q,v)$ integrability of this $|\nabla_{s,p,Q}^\tau f|$  (in order for $\nu$ to be finite on every cube). Note that also in this case, if this integrability does not hold, then the result we want to prove is trivial, as the right-hand side is infinite.

At this point, we are ready to perform our argument. Once we got the starting points \eqref{startingclassical} and \eqref{startingfractional},   we will apply Theorem \ref{thm:Lp(w)-a(Q)-clean-sht-euclidean} to the corresponding functionals 
\begin{equation}\label{funcionales}
a_{s,p}(Q):=\frac{[w,v]_{A_{q,p}^{s-\gamma,r}(\Omega)}\ell(Q)^\gamma}{w_r(Q)^\frac{1}{q}}\left(\int_Q|\nabla_{s,p,Q}^\tau f|(y)^p v(y) dy\right)^{1/p},\qquad s\in(0,1],
\end{equation}
where by an abuse of notation we will write  $|\nabla_{1,p,\Omega}^\tau f|:=|\nabla f|$   for any $p$. By doing this, we get, for any $s\in(0,1]$,
\[
\left(\frac{1}{w_r(Q)}\int_Q |f(x)-f_Q|^qw(x)dx\right)^{\frac{1}{q}}\leq C_{n,s,\alpha}a_{s,p}(Q),\qquad Q\subset \Omega,
\]
that is,
\[
\left(\int_Q |f(x)-f_Q|^qw(x)dx\right)^{\frac{1}{q}}\leq C[w,v]_{A_{q,p}^{s-\gamma,r}(\Omega)}\ell(Q)^\gamma\left(\int_Q |\nabla_{s,p,Q}^\tau f|(x)^p v(x)dx\right)^{1/p},
\]
for any $Q\subset\Omega$, where $C:= C_{n,s,\gamma}$. 

Once we have an estimate for any cube inside $\Omega$, we will focus on Boman chains of cubes of $\Omega$. As commented in Section  \ref{sec:preliminaries}, these cubes $W$ can be assumed to satisfy the Whitney property $d(x)\asymp \ell(W)$  for any $x\in W$. This will allow to replace the sidelenght of the cube in the estimate above by the distance to the boundary, and then, by multiplying both sides of the inequality by $\phi(\ell(Q))$, for $\phi$ a positive increasing function satisfying $\phi(2t)\leq C\phi(t)$, we obtain the following estimate on for any cube $W$  from a Boman chain
\[
\left(\int_W |f(x)-f_Q|^qw_\phi(x)w(x)dx\right)^{\frac{1}{q}}\leq C_{n,s,\gamma}\left(\int_W |\nabla_{s,p,\Omega}^\tau f|(x)^pw_{\Phi,\gamma p}(x)v(x)dx\right)^{1/p},
\]
where we recall that $w_\phi(x)=\phi(d(x))$  and $w_{\Phi,\gamma p}=d(x)^{\gamma p}w_\phi(x)^{\frac{p}{q}}$.

We can now apply  Theorem \ref{Chua2}.   Note that we only need to assume $w$ to be doubling as, in the argument in the proof of the chaining result (which we will outline in the following), we can replace  the improving weight $\phi(d(x))$ in the left-hand side by the sidelenght of each Whitney cube in the Boman chain of $\Omega$. This allows us to perform the argument in \cite{Chua1993} with the weight $w$ (that needs to be doubling\footnote{We would like to point out that the only step where the doubling property of the weight  is used is in the adapted chaining argument Theorem \ref{Chua2} which is just a modification of \cite[Lemma 2.8]{Chua1993}. In recent personal communications with the author of that work, we have discovered the existence of his new work \cite{Chua2018}, where he proves a quite general version of the chaining result which allows to obtain (from a starting inequality on balls) a Poincar\'e inequality in the whole domain just by asking $w$ to satisfy a somehow weak doubling property on certain balls. Our result is probably partially contained in his result once one has our starting points, and thus this shows that a stronger version of Theorem \ref{unified} could be obtained by considering this improved chaining result, avoiding this way the doubling condition on $w$.}) and then to recover the improving weight almost at the end of the proof. By doing this, we obtain the desired inequality from the inequality above, namely
\begin{equation}\label{Poindomain}
\inf_{c\in \mathbb{R}}\|f-c\|_{L^{q}(\Omega,w_\phi dw )}\lesssim [f]_{W_{\tau}^{s,p}\left(\Omega,w_{\Phi,\gamma p}dv \right)}.
\end{equation}
Note that, in the case $s<1$, the one-variable weight $w_{\Phi,\gamma p}$ can be replaced by the two-variables weight $v_{\Phi,\gamma p}(z,y)=\min_{x\in\{z,y\}}w_{\Phi,\gamma p}(x)$.
\end{proof}

\begin{rem}

 It should be noted that our result does not improve the main result in \cite{Drelichman2008} in the non-fractional case.  On one hand, if we do not want to ask $w$ to satisfy the $A_\infty$ condition, then  $w$ is somehow forced to satisfy (together with $v$), the condition $(w,v)\in A_{q,p}^{s-\gamma, r}$, for some number  $r$ strictly larger than $1$, in contrast with the result in \cite{Drelichman2008}, where the authors are able to consider the case in which $r=1$. Observe that the case $p=q$ in \cite{Drelichman2008} is improved by our result since we are able to take $r=1$ in the right-hand side integral in \eqref{DDp=q} and also we do not have to ask for any further condition on $v$. Note that, in our setting, the doubling condition on $w$ implies the reverse doubling condition. On the other hand, if we want to take $r$ to be $1$ in \eqref{condicion}, we so far have to ask $w$ to be in $A_\infty$, instead of asking for the reverse doubling property only, as they do in \cite{Drelichman2008}.  Finally we note that in contrast with the result in \cite{Drelichman2008}, we are able to plug more improving weights at both sides of our inequalities. 
\end{rem}
\begin{rem}
We now turn our attention  to the main result in \cite{Cejas2019}. First, we note that our result does not contain improving weights of the form $w_\phi^F(x)=\phi(d_F(x))$, where $d_F(x)=\inf_{y\in F}|x-y|$ for a compact subset $F\subsetneq\partial\Omega$. Also, if we want $w$ to not necessarily be in $A_\infty$, then we are somehow forced to work in the Euclidean space, as we do not know a more abstract counterpart of Theorem \ref{PRimprovement}.  Hence the comments we will give in the following will be enframed in the Euclidean setting. Even if we are not able to obtain this improving weights of the form $w_\phi^F$ depicted  above, we are able to obtain a quite large class of improving weights for which a \emph{weighted} improved fractional Poincar\'e inequality holds. Thus we extend the main result in \cite{Cejas2019} by adding weights to the final result.  
 \end{rem}

\section{Results in metric spaces for $A_\infty$ weights}\label{sec:metric}

As we know that Theorem \ref{weakimproving} is true for any space of homogeneous type, we can think of a generalization of Theorem \ref{unified} to this more general context (or at least to the context of doubling measure metric spaces). In order to do this, we have to redefine all the concepts we have worked with in the more general setting we are attempting to work in. If we succed in doing this, we will get a full generalization of the main result in \cite{Cejas2019} (up to the consideration of the fact the improving weights which include the parameter $F\subsetneq\partial\Omega$ are not included in our result) and the main result in \cite{Drelichman2008} (up to the fact that we will be asking $w$ to be in $A_\infty$). Our results will be based on Theorem \ref{weakimproving} and also in the ``weak implies strong'' argument we have mentioned above. Observe that this was not needed in the Euclidean case.

The fact that we are working on cubes of the Euclidean space is not fundamental except (to the best of our knowledge) for Theorem \ref{thm:Lp(w)-a(Q)-clean-sht-euclidean}.  Moreover, we know that similar (unweighted) results to Theorem \ref{unified} for the case $s<1$ make sense in the general setting of metric spaces with a doubling measure, as one can check in \cite{Cejas2019}. Even more, the result would make sense for a space of homogeneous type, that is, a space $(X,d,\mu)$, where $d$ is a quasimetric and $\mu$ is a doubling measure. Recall that a quasimetric $d$ on a set $X$ is a nonnegative function defined on $X\times X$ which satisfies 
\begin{enumerate}
\item $d(x,y)=0$ if and only if $x=y$;
\item $d(x,y)=d(y,x)$ for all $x,y\in X$;
\item There exists a finite constant $K\geq 1$ such that 
\[d(x,y)\leq K[d(x,z)+d(z,y)],\quad x,y,z\in X.\]
\end{enumerate}

Observe that the doubling property of $\mu$ gives (see \cite{Coifman1971}) that  $(X,d)$  has the following (geometric) doubling property:
There exists a positive integer $N\in \mathbb{N}$ such that, for every point $x\in X$ and for every $r>0$, the ball $B(x,r):=\{y\in X:d(x,y)<r\}$ can be covered by at most $N$ balls $B(x_i,r/2)$. Balls in this context are not necessarily open sets.

In this case, $W_{\tau}^{s,p}(\Omega,d\mu)$ and the seminorm $[f]_{W_{\tau}^{s,p}(\Omega,w_{\Phi,\gamma p}
v)}$, $0<s\leq1$, are defined in an analogous way to the Euclidean case by the (fractional) derivatives  
\begin{equation}\label{fractionalderSHT}|\nabla_{s,p,B}^\tau f|(y)
=\left(\int_{B^*\cap B(y,\tau r(B^*))}\frac{|f(x)-f(y)|^p }{\mu[B(y,d(x,y))]d(x,y)^{sp}}d\mu(x)\right)^{1/p}\chi_{\Omega}(y),
\end{equation}
where, as in the proof of Theorem \ref{unified}, by an abuse of notation  $|\nabla_{1,p,B}^\tau f| $ will be defined to be (for every $p$) the corresponding gradient in our context. For this we mean any function  $g$ with the truncation property (see \cite{Hajlasz2000} for details on the truncation property)  satisfying the $(1,1)$-Poincar\'e  inequality (as defined in \cite{Hajlasz2000}), i.e.
\begin{equation}\label{Poin11}
\dashint_B |f-f_B|d\mu\leq C r(B)\dashint_{\lambda B} gd\mu
\end{equation}
for any ball $B$ such that $\lambda B\in \Omega$, where $f\in L^1_{\loc}(X)$ and $g\in L^1(X)$, and $\lambda\geq1$, $C>0$ are fixed constants. In the literature, it is usual to consider $g$ to be an upper gradient of $f$. See \cite{Heinonen2015,Hajlasz2000} for  good references about Poincar\'e inequalities in metric spaces based on the use of upper gradients.

\begin{thm}\label{unifiedSHT}
Let $s\in(0,1]$ and $0\leq \gamma\leq s$. Let $(X,d,\mu)$ a metric space endowed with a doubling measure $\mu$ with doubling dimension $n_\mu$. Consider $1\leq p\leq q\leq \frac{n_\mu p}{n_\mu-(s-\gamma)p}$. Let $\Omega$ be  a bounded John domain and consider an  increasing function $\phi$  with $\phi(2t)\leq C\phi(t)$ such that $w_\phi\in L^1_{\loc}(\Omega)$.  Let $w$ a doubling weight and $v$ a weight. If $f\in W_{\tau}^{s,p}\left(\Omega,d\mu \right)$ for $\tau\in(0,1)$ and $(w,v)\in A_{q,p}^{s-\gamma,r}$ for some $r>1$, then 
\[
\inf_{c\in\mathbb{R}}\|f-c\|_{L^q(\Omega, w_\phi w)}\lesssim[f]_{W_{\tau}^{s,p}\left(\Omega,w_{\Phi,\gamma p}v\right)}.
\]
When $s<1$, the right hand side of the inequality above can be replaced by the quantity  $[f]_{W_{\tau}^{s,p}\left(\Omega,v_{\Phi,\gamma p}vd\mu \right)}$.

\end{thm}

\begin{proof}
In this setting, we have Theorem \ref{weakimproving}  at hand. Thus, we are left with obtaining suitable starting points. Consider a domain $\Omega$ in $X$. It is not difficult to see that the corresponding nonfractional starting point can be obtained in a similar way to \eqref{startingclassical} for any pair of functions $(f,g)$ satisfying the $(1,1)$-Poincar\'e  inequality (as defined in \cite{Hajlasz2000}), i.e.
\begin{equation}\label{Poin11}
\dashint_B |f-f_B|d\mu\leq C r(B)\dashint_{\lambda B} gd\mu
\end{equation}
for any ball $B$ such that $\lambda B\in \Omega$, where $f\in L^1_{\loc}(X)$ and $g\in L^1(X)$, and $\lambda\geq1$, $C>0$ are fixed constants.

By working as in the Euclidean case, (we are assuming $g$ to satisfy the truncation property) the starting point we get in this case is clearly
\begin{equation}\label{startingclassical1}
\dashint_B|f-f_B|d\mu\leq \frac{C[w,v]_{A_{q,p}^{1-\gamma,1}(\Omega,\mu)}r(B)^\gamma}{\lambda} \left(\frac{1}{w(\lambda B)^{\frac{p}{q}}}\int_{\lambda B} g^p vd\mu\right)^{\frac{1}{p}},
\end{equation}
for any ball $B$ such that $\lambda B\subset \Omega$.

Now we will try to obtain the starting point which corresponds to \eqref{startingfractional}. Let $B$ be a ball in  $(X,d,\mu)$. We will recall here the metric counterpart of Lemma \ref{HV}, which was already introduced in \cite[Lemma 4]{Cejas2019}.  For convenience, let us suppose $d$ to be a metric, so $K=1$. Let $1\leq p<\infty$ and let $s,\tau\in(0,1)$. Let us consider a covering $\mathcal{B}=\{B_i\}_{i\in J}$ of $B$ by $J$ balls of radious $\frac{\tau}{L}r(B)$ for some $L>2$.   
This can be done in such a way that, when $R$ is the union of two balls $B_i$ and $B_j$ with overlapping dilations  (i.e. with  $\lambda B_i\cap \lambda B_j\neq \emptyset$ for some $\lambda>1$ sufficiently small with respect to $L$), $R\subset B(y,\tau r(B))$ for every $y\in R$. Also, such an $R$ satisfies $R\subset B^*$ (for $B^*$ some dilation of $B$ by a factor larger than $2$) and $\mu[B(z,d(z,y))]\lesssim\mu(R)\asymp \mu(B)$ for every pair of points $y$ and $z$ in $R$. Observe that the index set $J$ is uniformly finite for every ball $B$, as $X$ satisfies the geometric doubling property.

Once we have this construction, observe that, for the union $R$ of two balls in $\mathcal{B}$ with overlapping dilations, we have, by the doubling condition and condition \eqref{condicion}

\begin{equation}\label{starting11}
\begin{aligned}
\dashint_R|f-f_R|d\mu&\leq \dashint_R\dashint_R|f(x)-f(y)|d\mu(x)d\mu(y) \\
&\leq \dashint_R\dashint_R|f(x)-f(y)| d\mu(x)v(y)^{\frac{1}{p}-\frac{1}{p}}d\mu(y)\\
&\leq \left(\dashint_R\dashint_R|f(x)-f(y)|^p v(y)d\mu(x)d\mu(y)\right)^{\frac{1}{p}}\left(\dashint_Rv^{1-p'}d\mu\right)^{\frac{1}{p'}}\\
&\leq \left(\dashint_R\int_{R}\frac{|f(x)-f(y)|^p}{\mu[B(y,d(x,y))]} d\mu(x)v(y)d\mu(y)\right)^{\frac{1}{p}}\left(\dashint_Rv^{1-p'}d\mu\right)^{\frac{1}{p'}}\\
&\leq r(B^*)^{s} \left(\dashint_R|\nabla_{s,p,B}^\tau f|(y)^pv(y)d\mu(y) \right)^{\frac{1}{p}}\left(\dashint_Rv^{1-p'}d\mu\right)^{\frac{1}{p'}}\\
&\lesssim\frac{r(B^*)^{s}}{\mu(B^*)^{\frac{1}{p}}}\left(\int_R|\nabla_{s,p,B}^\tau f|(y)^p v(y)d\mu(y)  \right)^{\frac{1}{p}}\left(\dashint_{B^*}v^{1-p'}d\mu\right)^{\frac{1}{p'}}\\
&\leq\frac{[w,v]_{A_{q,p}^{s-\gamma,r}(\Omega,\mu)}r(B^*)^\gamma}{w(B^*)^\frac{1}{q}}\left(\int_{B^*}|\nabla_{s,p,B}^\tau f|(y)^p v(y)d\mu(y)\right)^{\frac{1}{p}}.
\end{aligned}
\end{equation}

With this in mind, observe that, by Minkowski's,
\[
\begin{split}
\frac{1}{\mu(B)}\int_B|f(y)-f_{B}|d\mu(y)&\lesssim \frac{1}{\mu(B)}\int_B |f(y)-f_{B_1}|d\mu(y)\\
&\lesssim \sum_{j\in J}\frac{1}{\mu(B_j)}\int_{B_j}|f(y)-f_{B_j}|d\mu(y)\\
&\qquad+ \sum_{j\in J}\frac{1}{\mu(B_j)}\int_{B_j}|f_{B_j}-f_{B_1}|d\mu(y).
\end{split}
\]

The first sum is bounded by the quantity above, so it is enough to estimate the second sum. In order to do this, let us fix $B_j$, $j\in J$ and let $\sigma:\{1,2,\ldots,l\}\to J$, $l\leq \# J$ an injective map such that $\sigma(1)=1$ and $\sigma (l)=j$, and the subsequent balls $B_{\sigma(i)}$ and $B_{\sigma(i+1)}$ have overlapping dilations. Since $l\leq \# J$, we obtain
\[
\begin{split}
|f_{B_j}-f_{B_1}|&\leq \left(\sum_{i=1}^{l-1}|f_{B_{\sigma(i+1)}}-f_{B_{\sigma(i)}}|\right)\\
&\leq\sum_{i=1}^{l-1}|f_{B_{\sigma(i+1)}}-f_{B_{\sigma(i+1)}\cup B_{\sigma(i)}}|\\
&\qquad\qquad+ \sum_{i=1}^{l-1}|f_{B_{\sigma(i+1)}\cup B_{\sigma(i)}}-f_{B_{\sigma(i)}}|.
\end{split}
\]
The two sums above can be bounded in the same way, so we will just work with the first one. For each term we have
\[
\begin{split}
|f_{B_{\sigma(i+1)}}-&f_{B_{\sigma(i+1)}\cup B_{\sigma(i)}}|^q\\
&= \frac{1}{\mu(B_{\sigma(i+1)})}\int_{B_{\sigma(i+1)}}|f_{B_{\sigma(i+1)}}-f+f-f_{B_{\sigma(i+1)}\cup B_{\sigma(i)}}|^qd\mu\\
&\lesssim \frac{1}{\mu(B_{\sigma(i+1)})}\int_{B_{\sigma(i+1)}}|f-f_{B_{\sigma(i+1)}}|^qd\mu\\
&\qquad \qquad +\frac{1}{\mu(B_{\sigma(i+1)}\cup B_{\sigma(i)})}\int_{B_{\sigma(i+1)}\cup B_{\sigma(i)}}|f-f_{B_{\sigma(i+1)}\cup B_{\sigma(i)}}|^qd\mu,
\end{split}
\]
where we have used the conditions on the union of two balls of the covering with overlapping dilations and the doubling condition.
In the last two integrals we can apply the first estimate above and then the uniform finiteness of $\#J$ allows us to obtain the desired result, that is, the starting point

\begin{equation}\label{startingfractional1}
\begin{aligned}
\dashint_B|f-f_B|d\mu \lesssim\frac{[w,v]_{A_{q,p}^{s-\gamma,r}(\Omega,\mu)}r(B^*)^\gamma}{w(B^*)^\frac{1}{q}}\left(\int_{B^*} |\nabla_{s,p,B}^\tau f|(y)^p  v(y) d\mu(y)\right)^{\frac{1}{p}}.
\end{aligned}
\end{equation}

Let us write in general $B^*=\lambda B$, for some $\lambda\geq1$ (which, in the nonfractional case will be the $\lambda$ in the Poincar\'e inequality and in the fractional case will be needed to be larger than $2$). Then in both cases we have that
\[\dashint_B|f-f_B|\leq C a_{p,s}(B^*)\]
whenever $B^*\subset \Omega$, where $a_{p,s}$ is the analogous to the one in \eqref{funcionales}, defined by the (fractional) derivatives  $|\nabla_{s,p,B}^\tau f|$ in \eqref{fractionalderSHT}.

Then, as $K\widehat{B_0^*}\subset \Omega$ (let us write the following again in the general setting of spaces of homogeneous type) implies $B^* \subset \Omega$ for any $B$ such that $B\subset \widehat{B_0}$, then   for any ball $B_0$ such that $K\widehat{B_0^*}\subset \Omega$, we get the weak inequality 
\[\|f-f_{B_0}\|_{L^{q,\infty}(B_0,w)}\leq Cr(B^*)^\gamma \left(\int_{\widehat{B_0^*}} |\nabla_{s,p,\widehat{B_0}}^\tau f|(y)^p v(y)d\mu(y)\right)^{1/p}.\]

The weak implies strong argument (which also works for the fractional derivative in the context of spaces of homogeneous type \cite{Dyda2016}) gives us, from this, the strong inequality 
\[\left(\int_{B_0}|f-f_{B_0}|^q wd\mu\right)^{\frac{1}{q}}\leq Cr(B^*)^\gamma \left(\int_{\widehat{B_0^*}} |\nabla_{s,p,\widehat{B_0}}^\tau f|(y)^pv(y)d\mu(y)\right)^{1/p},\]
for any ball $B_0$ satisfying $K\widehat{B_0^*}\subset \Omega$ (i.e. very small balls in $\Omega$).  

If we take the Whitney decomposition given in \cite{Grafakos2014} and perform an argument like the one in \cite[Theorem 3.8]{Diening2010}, then we obtain a chain decomposition of a John domain $\Omega$ in $X$ built by using balls $\{B_i\}_{i\in \mathbb{N}}$ in $\Omega$ satisfying, for some values $c,\sigma,N\geq1$, the following Whitney-type properties:
\begin{enumerate}
\item $c^{-1}r(B_i)\leq d(x)\leq cr(B_i)$ for any $x\in B_i$, $i\in \mathbb{N}$;
\item and $\sum_{i\in\mathbb{N}} \chi_{\sigma B_i}\leq N\chi_\Omega$.
\end{enumerate}
 We just have to choose $\delta<1$ (recall that $\widehat{B}=(1+\delta)KB$) and a Whitney decomposition with balls so small that $\sigma\geq 2\lambda K^2$ (the balls $W$ in the Whitney covering will satisfy the Whitney property and $CW\subset \Omega$ for $C>2\lambda K^2$ and the balls in the Boman chain will be of the form $B=C/\sigma W$, so the balls $\sigma B$ are in $\Omega$ and will satisfy the Whitney property). Observe that $N$ can be quite large, depending on the preceding parameters. With this choice, each ball in this chain decomposition satisfies the conditions above and thus, we can perform exactly the same argument as in the Euclidean case. We can then use the obvious version of Theorem \ref{Chua2} for spaces of homogeneous type, obtaining
\[\inf_{c\in\mathbb{R}}\left(\int_{\Omega}|f-c|^q w_\phi wd\mu\right)^{\frac{1}{q}}\leq C \left(\int_{\Omega} |\nabla_{s,p,\Omega}^\tau f|(y)^p(y)w_{\Phi,\gamma p}(y)(y)v(y)d\mu(y)\right)^{1/p}.\]  

\end{proof}

\begin{rem}
Hence, we have obtained the result of the previous section in the more general context of spaces of homogeneous type, although we are assuming here $w$ to be in $A_\infty(\mu)$. Despite the fact that we are forced to consider the assumption $w\in A_\infty$, note that we can now take $r=1$ in the $A_{q,p}^{\gamma,r}$ condition. 
\end{rem}
\begin{rem}
All these computations can be performed in weak John domains, and also they can probably be performed for more general functions of $d(x)$ than $\phi(t)=t^{\gamma p}$ at the right-hand side (just by defining a more general version of the class $A_{q,p}^{\gamma,r}$).
\end{rem}

\section{Acknowledgements}

 We want to thank professor Seng Kee  Chua for pointing out us his result and also for his nice conversations about the topic. 

The  author is supported by the Basque Government through the BERC 2018-2021 program and by Spanish Ministry of Science, Innovation and Universities through BCAM Severo Ochoa accreditation SEV-2017-0718. He is also supported by MINECO through the MTM2017-82160-C2-1-P project funded by (AEI/FEDER, UE), acronym ``HAQMEC'' and through ''la Caixa'' grant.


\printbibliography

\end{document}